\documentclass[11pt,a4paper]{article}
\usepackage{hyperref}
\usepackage[english]{babel}
\usepackage[utf8]{inputenc}
\usepackage[T1]{fontenc}
\usepackage{amsmath,amsthm,amssymb,amsfonts}
\numberwithin{equation}{section}
\usepackage{enumerate}
\usepackage{faktor}
\usepackage{dsfont}
\usepackage{scalerel,stackengine}
\usepackage{textcomp}
\usepackage{graphicx}
\usepackage{extarrows}
\usepackage{epigraph}
\usepackage{graphicx}
\usepackage{subcaption}
\usepackage{sidecap}
\usepackage{indentfirst}
\usepackage{tikz}
\usepackage{tikz-cd}
\usetikzlibrary{shapes.misc,arrows,matrix,patterns,decorations.markings,positioning}
\tikzset{cross/.style={cross out, draw=black, minimum size=2*(#1-\pgflinewidth), inner sep=0pt, outer sep=0pt},
cross/.default={4.5pt}}
\usepackage{pgfplots}
\usepackage{caption}
\usepackage{float}
\usepackage{color}
\usepackage{epstopdf}
\usepackage{epsfig}
\usepackage{stmaryrd}
\usepackage{color}
\usepackage{geometry}
\usepackage{multicol}
\usepackage{verbatim}

\DeclareMathOperator{\SL}{SL}

\DeclareMathOperator{\spinc}{Spin^c}

\DeclareMathOperator{\tb}{tb}
\DeclareMathOperator{\TB}{TB}
\DeclareMathOperator{\rot}{rot}
\DeclareMathOperator{\slk}{sl}

\DeclareMathOperator{\HFK}{HFK}
\DeclareMathOperator{\HFL}{HFL}

\DeclareMathOperator{\CFL}{CFL}

\DeclareMathOperator{\HFKhat}{\widehat{HFK}}
\DeclareMathOperator{\HFLhat}{\widehat{HFL}}
\DeclareMathOperator{\SFH}{SFH}
\DeclareMathOperator{\HFhat}{\widehat{HF}}
\DeclareMathOperator{\CF}{CF}
\DeclareMathOperator{\HF}{HF}
\DeclareMathOperator{\Leg}{\mathfrak{L}}
\DeclareMathOperator{\Leghat}{\widehat{\mathfrak{L}}}
\DeclareMathOperator{\EH}{EH}
\DeclareMathOperator{\EHdirect}{\underrightarrow{\EH}}
\DeclareMathOperator{\EHinverse}{\underleftarrow {\EH}}

\DeclareMathOperator{\Ker}{Ker }

\DeclareMathOperator{\lk}{\ell k}

\DeclareMathOperator{\rk}{rk}

\renewcommand{\geq}{\geqslant}
\renewcommand{\leq}{\leqslant} 
\renewcommand{\epsilon}{\varepsilon}

\newcommand{\N}{\mathbb{N}}
\newcommand{\Z}{\mathbb{Z}}
\newcommand{\Q}{\mathbb{Q}}
\newcommand{\F}{\mathbb{F}}

\DeclareFontFamily{U}{mathx}{\hyphenchar\font45}
\DeclareFontShape{U}{mathx}{m}{n}{
      <5> <6> <7> <8> <9> <10>
      <10.95> <12> <14.4> <17.28> <20.74> <24.88>
      mathx10
      }{}
\DeclareSymbolFont{mathx}{U}{mathx}{m}{n}
\DeclareFontSubstitution{U}{mathx}{m}{n}
\DeclareMathAccent{\widecheck}{0}{mathx}{"71}
\DeclareMathAccent{\wideparen}{0}{mathx}{"75}

\newtheorem{thm}{Theorem}[section]
\newtheorem*{thm*}{Theorem}
\newtheorem{lemma}[thm]{Lemma}
\newtheorem{prop}[thm]{Proposition}
\newtheorem*{prop*}{Proposition}
\newtheorem{cor}[thm]{Corollary}
\newtheorem{conj}[thm]{Conjecture}

\theoremstyle{definition}

\newtheorem{rmk}[thm]{Remark}
\newtheorem*{rmk*}{Remark}
\newtheorem{convention}[thm]{Convention}

\usepackage{xpatch}
\makeatletter
\xpatchcmd{\@thm}{\thm@headpunct{.}}{\thm@headpunct{}}{}{}
\makeatother

\pgfplotsset{compat=1.13}
\begin{document}
\title{Legendrian invariants and half Giroux torsion}
\author{Alberto Cavallo and Irena Matkovi\v{c}\\ \\
 \footnotesize{Institute of Mathematics of the Polish Academy of Sciences (IMPAN),}\\ 
 \footnotesize{Warsaw 00-656, Poland}\\ \\
 \footnotesize{Uppsala Universitet,}\\
 \footnotesize{Uppsala 751 06, Sweden}\\ \\  \small{acavallo@impan.pl}\hspace{2cm}\small{irena.matkovic@math.uu.se}}
\date{}

\maketitle
\begin{abstract}
We collect some observations about Legendrian links with non-vanishing contact invariants, mostly concerning the non-loose realizations of links and the addition of boundary-parallel half Giroux torsion.

In particular, we show that every null-homologous link with irreducible complement admits a non-loose Legendrian realization with non-zero (at least) invariant $\EH$ in $\SFH$, in some overtwisted contact structure (determined by its gradings); for many links these come from Gabai's work, for others the existence follows from the sutured interpretation of link Floer invariants. 

We reveal that separating half Giroux torsion does not necessary cause Legendrian invariants to vanish. Furthermore, we propose a conjectural characterization of links with non-vanishing $\Leghat$ in $\HFLhat$ among links with non-zero $\Leg$ in $c\HFL^-$.
 
 \bigskip
 \small{2020 {\em Mathematics Subject Classification}: 57K33, 57K18.}
\end{abstract}

\section{Introduction}
The Legendrian and transverse knot invariants have been defined in various versions of Heegaard Floer theory; in particular, in sutured Floer homology \cite{HKM.eh}, and every (minus, hat and plus) version of knot Floer homology \cite{LOSSz}. Various constructions (via convex decompositions, Heegaard diagrams, braid representation, or grids) have been proved to be equivalent \cite{SV, G.i, EVVZ, BVVV}. 

On the other hand, it has also been observed that the information from different invariants is not the same. Indeed, there are knots with non-zero $\EH$-invariant in $\SFH$ but vanishing $\Leg$ in $\HFK^-$; for example, the non-loose unknots \cite{EVVZ}. There are knots with non-zero $\Leg$ in $\HFK^-$, but zero $\Leghat$ in $\HFKhat$; as, for example, all Legendrian knots in $(S^3,\xi_\text{std})$ which are positive stabilizations \cite{LOSSz}. There are knots with non-zero $\Leghat$ in $\HFKhat$, but zero $\Leg^+$ in $\HFK^+$; in particular, the bindings of open books supporting overtwisted contact manifolds, by combining \cite{VV} and \cite{M.s}. From the sutured interpretation \cite{SV, G.i,EVVZ}, it is clear though that non-zero $\Leg^+$ implies non-zero $\Leghat$ implies non-zero $\Leg$ implies non-zero $\EH$. Our motivation is to understand how different is the contact topological behavior of Legendrian links with non-zero invariant $\Leg^+$ or $\Leghat$ compared to the ones with non-zero only $\Leg$. 

\bigskip
We know that all overtwisted structures admit non-loose links; for example, the binding of any supporting open book is necessary non-loose (in fact, its complement is universally tight by \cite{ET}). The existence of a non-loose realization for every null-homologous knot type with irreducible complement was, on the other hand, formally proved by Etnyre \cite[Theorem 1.8]{E}; although there is a missing assumption in the statement of his theorem about the knot complement being irreducible, which is needed for it to be taut (now expanded in \cite{CEMM}). The same argument works for links as well. However, the 3-dimensional invariants of structures, which support non-loose realizations of a particular link, have not been identified in many cases. 
\medskip

\noindent\emph{Unknots.} In \cite{EF} Eliashberg and Fraser showed that there was a unique overtwisted structure on $S^3$ containing non-loose unknots; moreover, they completely classified them. 
\begin{convention}
 We normalize the 3-dimensional invariant $d_3$ using the convention of Ozsv\'ath and Szab\'o in \cite{OSz.c}. Hence, we assign $d_3$ equal to zero for structures inducing 2-plane fields homotopic to the Hopf fibration, and have $d_3(\xi)=1$ when $\xi$ is the structure that contains the non-loose unknots.
\end{convention}
\noindent\emph{Hopf links.}
 The Legendrian representatives of the Hopf links were completely classified by Geiges and Onaran in \cite{GO}; they live in overtwisted structures with $d_3=0,1,2$.
\medskip

\noindent\emph{Torus knots.}
 Non-loose torus knots were eventually completely classified by Etnyre, Min and Mukherjee in \cite{EMM}. Even so, we have no closed expression for overtwisted structures inhabited by non-loose $T_{p,q}$, only bounds on the values of $d_3$ where they appear \cite{M.n,EMM}.
\medskip

\noindent\emph{Fibered links.}
From the work of Etnyre and Vela-Vick \cite{VV,EVV} it follows that every fibered link in $S^3$ has a transverse representative with non-vanishing $\widehat\Leg$ in the contact structure it supports; namely, as the binding of the corresponding open book. The 3-dimensional invariant $d_3$ of the structure $\xi$ where such representative lives can be computed using link Floer homology: say $M_{\text{top}}$ and $A_{\text{top}}=g_3+n-1$ are the gradings of the unique top Alexander grading non-zero element of $\widehat\HFL$ of the mirror image of the link; then $d_3(\xi)=2A_{\text{top}}-M_{\text{top}}+1-n$, where $n$ is the number of components of the link, see Lemma \ref{lemma:tower}. 

If the link is not strongly quasi-positive (in other words, if it does not support the tight structure) then it is non-loose itself. In either case, we have two infinite families of non-loose realizations (by Etnyre \cite[Theorem 1.7]{E}), living in $d_3=2g_3-M_{\text{top}}+n-1$ and $d_3=1-M_{\text{top}}$; where for strongly quasi-positive links $M_{\text{top}}=2g_3+n-1$.

\subsection*{Non-loose realizations from taut foliations}
Etnyre in \cite[Theorem 1.8]{E} observes that taking the (perturbation of) taut foliation on a knot complement and gluing to it Giroux torsion layers gives two infinite families of transverse non-loose representatives with the self-linking number equal to $\pm(2g_3-1)$.
We first look at what additional Heegaard Floer invariants tell us about this construction. 

\begin{thm}
 \label{thm:taut}
 If a contact structure $\xi_L$ on $(Y(L),\Gamma_\infty)$ arises from a taut sutured manifold hierarchy, then it is the $\infty$-completion of a Legendrian link $\mathcal L$ with non-zero $\Leghat(\mathcal L)\in\HFLhat(-Y,L)$ and its Alexander grading equals $A(\Leghat(\mathcal L))=\frac{n+||L||_T}{2}$, where $||L||_T$ is the Thurston norm of $L$.
\end{thm}
Theorem \ref{thm:taut} follows from the work of Honda, Kazez and Mati\'c, as detailed in Section \ref{ss:taut}; Tovstopyat-Nelip informed us that in an ongoing joint work with Hedden, they also observe this result.

We can now express the 3-dimensional invariant of overtwisted structures supporting the two families of non-loose representatives in terms of Alexander and Maslov grading of the Legendrian invariant $\Leghat(\mathcal L)$, see Subsections \ref{subsection:AM} and \ref{subsection:Gtor}.

\begin{prop}
 \label{prop:taut}
 Let $L$ be a null-homologous $n$-component link, different from the unknot, in a rational homology $3$-sphere $Y$, whose complement $Y\backslash L$ is irreducible. Then there exists a Legendrian realization $\mathcal L$ as in Theorem \ref{thm:taut}. 
 Furthermore, there are (infinitely many) transverse non-loose realizations of $L$ with self-linking number $\slk$ equal to $||L||_T$ in $d_3(\xi)=1+||L||_T-M(\Leghat(\mathcal L))$, and equal to $-||L||_T$ in $d_3(\xi')=1-M(\Leghat(\mathcal L))$.
\end{prop}
In the 3-sphere depending on the $\HFL$-\emph{thickness} of a link, that is, one less than the number of lines in the $(M,A)$-plane where $\widehat{\HFL}$ is supported (see, for example, \cite{SSz}), and the link invariants $\tau$ and $\tau^*$ defined in \cite{Cavallo} we can greatly limit what the value of $d_3$ is.
\begin{cor}
 \label{cor:d3}
 Suppose that the link $L$ as before has thickness equal to $\emph{th}(L)$; then the two overtwisted structures on $S^3$ satisfy 
 \[\left|d_3(\xi)+\tau^*(L)+\dfrac{n-||L||_T}{2}-1\right|\leq\emph{th}(L)\hspace{0.5cm}\text{ and }\hspace{0.5cm}\left|d_3(\xi)+\tau^*(L)+\dfrac{n+||L||_T}{2}-1\right|\leq\emph{th}(L)\]
 In particular, if $L$ is $\HFL$-thin, which means $\emph{th}(L)=0$, then we have $d_3(\xi)=\frac{1+||L||_T+\sigma(L)}{2}$ and $d_3(\xi')=\frac{1-||L||_T+\sigma(L)}{2}$.    
\end{cor}
Comparing the Alexander gradings (see Subsection \ref{subsection:AM}) immediately reveals that, unless the link satisfies the (classical) Thurston-Bennequin equality $\SL(L)=||L||_T$ where $\SL$ denotes the maximal self-linking number of $L$, the Legendrian link $\mathcal L$ built as perturbation of a taut foliation necessary lives in an overtwisted structure. Indeed, for $\mathcal L$ we have $\tb(\mathcal L)-\rot(\mathcal L)=||L||_T$ which by assumption cannot be achieved in $\xi_\text{std}$. 
For the link types that satisfy the equality, we cannot say whether $\mathcal L$ lives in $\xi_\text{std}$, except when $L$ is fibered in which case the construction gives the binding of the ambient structure it supports.

\begin{rmk}
 The precise definition of the Thurston norm $||L||_T$ can be found in \cite{Cavallo3}, together with a proof that for a non-trivial link $L$, with irreducible complement, one has that $||L||_T$ is equal to minus the maximal Euler characteristic of a (possibly disconnected) Seifert surface of $L$.
\end{rmk}
 
\subsection*{Non-loose realizations from knots fulfilling Thurston-Bennequin equality}
However, for the links for which the $\tau_\xi$-Bennequin inequality is sharp for some $\xi$  
we observe other non-loose Legendrian realizations, and the 3-dimensional invariant of their ambient is easily determined. Essentially, they are again all obtained by adding boundary parallel $\pi$-torsion (to the Legendrian link fulfilling the equality), but we might need destabilizations as we will describe in Section \ref{section:three}. Note then that our construction does not necessary give transversely non-loose links; for example, it does cover the non-loose Legendrian unknots. 

Anyway, the two constructions possibly coincide in the case of links satisfying the (classical) Thurston-Bennequin equality; such as strongly quasi-positive links. This is indeed the case for strongly quasi-positive bindings; for example, the positive Hopf link in $d_3=0$ and the positive torus knots $T_{p,q}$ in $d_3=1-(p-1)(q-1)$.

The main result is stated in terms of the invariant $\tau_\xi(L)$ of $L\hookrightarrow Y$ extracted from $c\HFL^-(-Y,L,\mathfrak t_\xi)$, and defined in \cite{Hedden_xi} by Hedden; the exact definition is given in Subsection \ref{subsection:AM}. 

\begin{thm}
 \label{thm:main}
 In every rational homology contact $3$-sphere $(Y,\xi)$ with non-zero contact invariant $\widehat c(Y,\xi)$, an $n$-component null-homologous link $L$ with a Legendrian representative $\mathcal L$ in $(Y,\xi)$ such that \[\tb(\mathcal L)-\rot(\mathcal L)=2\tau_\xi(L)-n\] has a non-loose realization with non-zero $\EH$-invariant in the overtwisted structure $\eta$ on $Y$, determined by $\mathfrak{t}_\eta=\mathfrak{t}_\xi\in \spinc(Y)$ and  $d_3(\eta)=d_3(\xi)-2\tau_\xi(L)+n$.
\end{thm}
Note that $\mathfrak t_\xi$ and $d_3(\xi)$ can be read from Heegaard Floer gradings of $\widehat c(Y,\xi)$.

Considering together Theorem \ref{thm:taut} for links that do not fulfill the $\tau_\xi$-Bennequin inequality (for any $\xi$) and Theorem \ref{thm:main} for all the others, we obtain a general existence result.
\begin{cor}
In every rational homology $3$-sphere where every tight structure has non-zero contact invariant, every null-homologous link with irreducible complement admits a non-loose realization with non-zero $\EH$-invariant.
\end{cor}

In \cite{Cavallo} we study a concordance invariant of a smooth link in $S^3$ called $\tau$-set of $L$; in particular, it was shown that such a set contains two main numerical concordance invariants called $\tau(L)$ and $\tau^*(L)$ that satisfy $\tau^*(L)=-\tau(L^*)$; 
note that when $K$ is a knot then $\tau(K)=\tau^*(K)$.
It is proved in \cite{Cavallo} that, whenever $\mathcal L$ is an $n$-component Legendrian link in $S^3$ with link type $L$, if the Legendrian invariant $\Leg(\mathcal L)$ is non-torsion then it belongs to the unique infinite cyclic subgroup $\F[U]$ of $c\HFL^-(L^*)$ with generator in bigrading $(-2\tau^*(L^*)+1-n,-\tau^*(L^*))$. We then have the following relation: \[\tau_{\xi_{\text{std}}}(L)=-\tau^*(L^*)=\tau(L)\] which implies that Theorem \ref{thm:main} can be restated in an easier way for links in $S^3$.

\begin{thm}
 Every $n$-component link $L\hookrightarrow S^3$ with a Legendrian representative $\mathcal L$ in $\xi_{std}$ such that \[\tb(\mathcal L)-\rot(\mathcal L)=2\tau(L)-n\] has a non-loose realization, with non-zero $\EH$-invariant, in the overtwisted structure $\eta$ with $d_3(\eta)=-2\tau(L)+n$. In particular, this is the case for every quasi-positive link. 
\end{thm}

We think (see Subsection \ref{ss:^vs-}) that preserving tightness by adding boundary parallel $\pi$-torsion (and possibly destabilizing) is potentially a complete characterization of links with non-vanishing $\Leghat$ among links with non-vanishing $\Leg$. This is indeed the case for negative torus knots.
We recall that a Legendrian knot is \emph{strongly non-loose} when its complement has no Giroux $\pi$-torsion.
\begin{prop}\label{prop:T(p,-q)}
 The complement of a strongly non-loose Legendrian negative torus knot $\mathcal L$ remains tight after adding boundary-parallel (half) Giroux torsion, if and only if either $\Leghat(\mathcal L)$ or $\Leghat(-\mathcal L)$ does not vanish. 
\end{prop}

This completes the Legendrian realization problem for knot Floer homology of negative torus knots, studied in \cite{M.n}.
\begin{cor}\label{cor:T(p,-q)}
The only negative torus knots $T_{p,-q}$ for which the torsion part of the knot Floer homology $\HFK^-(T_{p,q})$ is generated by Legendrian invariants are the alternating ones, that are $T_{2,-2n-1}$ for all $n\geq1$.
\end{cor}

\paragraph{Note on sublinks.}
Since every smooth knot can be perturbed to be Legendrian inside its tubular neighborhood, if a link $L$ contains a proper sublink, which satisfies the conditions in Theorem \ref{thm:main}, then $L$ itself will have a non-loose realization. In fact, if the complement $Y\setminus J$ of a Legendrian representative of $J$ is tight (or has non-zero $\EH$-invariant), and $L$ has $J$ as a sublink, then $Y\setminus L\hookrightarrow Y\setminus J$ and then the same has to hold also for the complement of $L$.
\begin{cor}
 \label{cor:sublink}
 If a link $L\hookrightarrow Y$ has a sublink $J$ of $m$ components which satisfies the hypothesis in Theorem \ref{thm:main} for a contact structure $\xi$ then $L$, as well as $J$, admits a non-loose realization in the overtwisted structure $\eta$ on $Y$, determined by $\mathfrak{t}_\eta=\mathfrak{t}_\xi\in \spinc(Y)$ and  $d_3(\eta)=d_3(\xi)-2\tau_\xi(J)+m$.
\end{cor}
In general, non-loose representatives for different sublinks, satisfying Theorem \ref{thm:main}, will produce non-loose representatives of $L$ living in different overtwisted structures.

On the other hand, from the construction we see (in Remark \ref{rmk:sub}) that with $\mathcal L$ as in Theorem \ref{thm:main} we have for each of its null-homologous Legendrian sublink $\mathcal J\subset\mathcal L$ an additional non-loose realization of L (with non-zero $\EH$-invariant) in $d_3(\eta_{J})=d_3(\xi)-(\tb(\mathcal J)-\rot(\mathcal J))$. However, unless $J$ satisfies the hypothesis in Theorem \ref{thm:main}, there might be no non-loose realization of $J$ in $d_3(\eta_{J})$.

Everything we said about sublinks apply also to the construction from taut sutured hierarchy, except the computation of the 3-dimensional invariant in the corollary.

\medskip
\subsection*{Non-vanishing of contact invariants in presence of half Giroux torsion}
We are particularly interested in the non-loose Legendrian knots $T_{2,2n+1}$ in the overtwisted $(S^3,\xi_n)$ with 3-dimensional invariant $d_3(\xi_n)=1-2n$, that is, the knots which were first described by Lisca, Ozsv\'ath, Stipsicz and Szab\'o as $L(n)$ in \cite[Figure 9]{LOSSz}. In \cite{LOSSz}, they proved these knots have non-vanishing $\widehat\Leg$; thus, they are transversely non-loose and we need no destabilizations in our construction (in Subsection \ref{ss:^vs-}) to ensure tightness. (The fact that $L(n)$ as transverse knots can be obtained by adding boundary parallel $\pi$-torsion to the maximal Thurston-Bennequin number Legendrian representatives in $\xi_{\text{std}}$ was observed already in \cite{EMM}. Note though that Legendrian knot $L(n)$ does not contain $\pi$-torsion; in fact, only its fifth negative stabilization does.)

\begin{rmk}
In \cite[Section 1.4]{EVVZ} Etnyre, Vela-Vick and Zarev define vanishing slope $\text{Van}^-(\mathcal K)$ for a null-homologous Legendrian knot $\mathcal K$ with a given Seifert framing and non-zero $\EH$-invariant as
\[\text{Van}^-(\mathcal K)=\sup\{(Y(K),\xi'_{\mathcal K}) \text{ negatively extends } (Y(K),\xi_{\mathcal K}), \EH(Y(K), \xi'_{\mathcal K})\neq0\}.\]
It is given as a pair $(-n,r)\in\Z_{\geq0}\times(\Q\cup\infty)$ where $n$ measures the amount of $\pi$-torsion and $r$ equals the boundary (dividing) slope. It is clear this invariant is only interested for knots with vanishing $\Leg$, when it is smaller than $(0,\infty)$, and the ones with non-vanishing $\Leg^+$, when it is greater than $(0,\infty)$; in all other cases it equals $(0,\infty)$.
In this terminology, the above paragraph shows that $\text{Van}^-(L(n))=(-1,\infty)$.
\end{rmk}

We know that all contact invariants of overtwisted structures vanish \cite{OSz.c}. Moreover, they vanish already in the presence of Giroux $2\pi$-torsion \cite{GHVHM}. The described example, however, exposes that the $\pi$-torsion is not enough for contact invariants to vanish, even if it appears (as for a knot complement necessary) along a separating torus.
\begin{prop}
 There exist Legendrian knots in $S^3$ with non-vanishing $\Leghat$-invariant, but whose complement contains half Giroux torsion.
\end{prop}
Therefore, we have the following corollary, which, in fact, could be seen already from the Stein fillability of the canonical contact structure on $T^3$ \cite{E1}, where the $\pi$-torsion is non-separating.
\begin{cor}
 The $\EH$-invariant of the half Giroux torsion layer $\big(T^2\times [0,1], \Ker(\sin(\pi t)\mbox{d} x+\cos(\pi t)\mbox{d} y\big)$ is non-zero.
\end{cor}

\paragraph*{Overview}
In Section \ref{section:two} we briefly review the different versions of Legendrian invariant, focusing on relation between its non-vanishing and grading (as needed for Proposition \ref{prop:taut} and in Theorem \ref{thm:main}). In Section \ref{section:three} we recall the sutured reinterpretation of invariants, and prove non-vanishing for taut foliation induced contact structures, which implies Theorem \ref{thm:taut}; we recall the definition of Giroux torsion in a contact 3-manifold and how it is related to Legendrian invariants, this is the key part in the proof of Theorem \ref{thm:main}. Finally, in Section \ref{section:four} we show some particular examples of presented constructions, for quasi-positive links, for fibered links, and further comment on sublinks.

\paragraph*{Acknowledgements}
We owe it to a question of Paolo Ghiggini to think about the behavior of contact invariants in the presence of half torsion. We thank Marco Golla, Georgios Dimitroglou Rizell, Antonio Alfieri and Steve Boyer for related conversations.

AC is partially supported by the NCN, under the project "Selected topics in knot theory" led by Maciej Borodzik at IMPAN. IM is supported by the Knut and Alice Wallenberg Foundation through the grant KAW 2021.0191, and by the Swedish Research Council through the grant number 2020-04426, as a post-doctoral researcher at the Uppsala Universitet. We are grateful to the Alfr\'ed R\'enyi Institute of Mathematics in Budapest for its hospitality during the semester on "Singularities and low-dimensional topology", and for the support from \'Elvonal (Frontier) grant KKP126683 (given by NKFIH).

\section{Legendrian invariants of links}
\label{section:two}
\subsection{Link Floer complexes}
We will need to understand the algebraic structure of link Floer complexes, as introduced by Ozsv\'ath and Szab\'o in \cite{OSlinks}, and further studied by the first author in \cite{Cavallo1}.

Assume that $Y$ is a rational homology 3-sphere and denote by $n$ the number of components of a null-homologous smooth link $L\hookrightarrow Y$. The bigraded chain complex $c\CFL^-(D)$ is defined from a multi-pointed \emph{Heegaard diagram} $D=(\Sigma,\alpha,\beta,\textbf w,\textbf z)$ as in \cite{OSlinks}: it is a finite dimensional free $\F[U]$-module, over the set of the intersection points of $D$, where $\F$ is the field with two elements; and its differential $\partial^-$ avoids the basepoints in $\textbf z$ and counts the multiplicities of the $\textbf w$'s. Such a complex factors through the $\spinc$-structures $\mathfrak t_1,...,\mathfrak t_t$ on $Y$. 
In the rest of the paper we denote the corresponding summand just by $c\CFL^-(Y,L,\mathfrak t)$, assuming that the diagram used to construct it is fixed a priori. 
The homology of $c\CFL^-(Y,L,\mathfrak t)$ is a finite dimensional $\F[U]$-module, called \emph{collapsed link Floer homology}, whose isomorphism type is a smooth link invariant of the link $L$, as proved in \cite{Book}.

By definition $c\CFL^-(Y,L,\mathfrak t)$ is a subcomplex of $c\CFL^\infty(Y,L,\mathfrak t):=c\CFL^-(Y,L,\mathfrak t)\otimes\F[U,U^{-1}]$, the latter being the graded object associated to the complex studied in \cite{Cavallo1}. We can then introduce two new complexes as follows \[c\CFL^+(Y,L,\mathfrak t):=\dfrac{c\CFL^\infty(Y,L,\mathfrak t)}{U\cdot c\CFL^-(Y,L,\mathfrak t)}\hspace{1cm}\text{ and }\hspace{1cm}\widehat{\CFL}(Y,L,\mathfrak t):=\dfrac{c\CFL^-(Y,L,\mathfrak t)}{U\cdot c\CFL^-(Y,L,\mathfrak t)}\:;\] we have that $c\CFL^+(Y,L,\mathfrak t)$ is an infinitely generated $\F[U]$-module, while $\widehat{\CFL}(Y,L,\mathfrak t)$ is a finitely generated $\F$-vector space (the $U$-action is trivial), see \cite{Book}. Hence, we have the following bigraded chain maps 
\begin{equation}
 \label{eq:maps}
 i^-:c\CFL^-(Y,L,\mathfrak t)\longrightarrow\widehat{\CFL}(Y,L,\mathfrak t)\hspace{1cm}\text{ and }\hspace{1cm}i^+:\widehat{\CFL}(Y,L,\mathfrak t)\longrightarrow c\CFL^+(Y,L,\mathfrak t)
\end{equation}
which induce bigraded homomorphisms $i^-_*$ and $i^+_*$ in homology.

A non-zero element in $c\CFL^-(Y,L,\mathfrak t)$ has a well-defined bigrading $(M,A)$, where $M$ is the \emph{Maslov} and $A$ the \emph{Alexander grading}, see \cite{OSlinks}; and multiplication by $U$ drops $M$ by 2 and $A$ by 1. It follows from \cite{Book} that there is a decomposition as direct sum of cyclic subgroups as follows: \[c\HFL^-(Y,L,\mathfrak t)\cong\F[U]_{(d_1,s_1)}\oplus\cdot\cdot\cdot\oplus\F[U]_{(d_{\ell},s_{\ell})}\oplus \text{Tors}\hspace{0.5cm}\text{ with }\hspace{0.5cm}\ell=2^{n-1}\cdot\rk\widehat\HF(Y,\mathfrak t)\:,\] where $\widehat\HF(Y,\mathfrak t)$ is the Heegaard Floer homology of $Y$ in the $\spinc$-structure $\mathfrak t$ (\cite{OS}) and the integers $-s_1,...,-s_{\ell}$ coincide with the elements of the $\tau$-set of $L$, the latter being the link concordance invariant introduced in \cite{Book}.
\begin{rmk}
 Note that the homology groups $c\HFL^-$ and $c\HFL^+$ coincide with $\HFK^-$ and $\HFK^+$ in the case of knots. 
\end{rmk}

\subsection{The Legendrian invariant and its grading}
\label{subsection:AM}
We survey the construction of Legendrian invariant in link Floer homology, as originally defined for knots by Lisca, Ozsv\'ath, Stipsicz and Szab\'o \cite{LOSSz}, and extended to links by the first author \cite{Cavallo2}. Particular attention is paid to its grading, and the conditions under which it survives through maps from $c\HFL^-$ to $\widehat{\HFL}$ and further to $c\HFL^+$.

Given a Legendrian $n$-component link $\mathcal L$ in $(Y,\xi)$, with link type $L$, we can find an open book decomposition $(B,\pi)$ for $Y$ such that $\mathcal L$ lies on the page $S_1=\pi^{-1}(1)$, see \cite{LOSSz} for the precise definition. Moreover, one can associate a multi-pointed Heegaard diagram to $(B,\pi)$: this diagram is called a \emph{Legendrian Heegaard diagram} and it is denoted by $D_{(B,\pi,A)}$, where $A$ is a collection of pairs of arcs in $S_1$ as described in \cite{LOSSz,Cavallo2}. The surface $\Sigma$ in $D_{(B,\pi,A)}$ is obtained by gluing the pages $S_1$ and $S_{-1}:=\pi^{-1}(-1)$ together; moreover, all the basepoints are contained in $S_1$.

Applying Heegaard Floer theory to the diagram $D_{(B,\pi,A)}$ yields a bigraded 
chain complex $c\CFL^-(D_{(B,\pi,A)})$, representing the link $L$ in the manifold $-Y$ (with reversed orientation), and its homology is then $c\HFL^-(-Y,L,\mathfrak t_\xi)$, see \cite{OSz.c}, where $\mathfrak t_\xi$ is the $\spinc$-structure induced by $\xi$ on $Y$.
As before, from now on we suppose that $D_{(B,\pi,A)}$ is fixed and we just write $c\CFL^-(-Y,L,\mathfrak t_\xi)$ for the chain complex defined from the Legendrian Heegaard diagram.

Furthermore, there is only one cycle in $c\CFL^-(-Y,L,\mathfrak t_\xi)$ 
that lies on the page $S_1$: this cycle is denoted by $x(\mathcal L)$ and its homology class $\mathfrak L(\mathcal L):=[x(\mathcal L)]$ in $c\HFL^-(-Y,L,\mathfrak t_\xi)$ is a Legendrian invariant of the link $\mathcal L$, see \cite{HKM.eh,LOSSz,Cavallo2}. 
Different versions of the Legendrian invariant $\Leg(\mathcal L)$ can be defined using the maps in Equation \eqref{eq:maps}: we call $\Leghat(\mathcal L):=i^-_*(\Leg(\mathcal L))$ and $\Leg^+(\mathcal L):=i^+_*(\Leghat(\mathcal L))$. In particular, when $\Leg^+(\mathcal L)$ is non-vanishing then both $\Leg(\mathcal L)$ and $\Leghat(\mathcal L)$ are non-vanishing. 

From \cite{Cavallo2} we have the chain map \[F:c\CFL^-(-Y,L,\mathfrak t_\xi)\longrightarrow\widehat\CF(D_{(B,\pi,A)},\mathfrak t_\xi)\] defined by setting $U=1$ which induces \[F_*:c\HFL^-(-Y,L,\mathfrak t_\xi)\longrightarrow\widehat\HF(-Y,\mathfrak t_\xi)\otimes(\F_{(-1)}\oplus\F_{(0)})^{\otimes n-1}\:.\] We have that the kernel of $F_*$ coincides with the torsion subgroup, while $F_*$ collapses each of the towers $\F[U]_{(d_i,s_i)}$ of $c\HFL^-(-Y,L,\mathfrak t_\xi)$ to exactly one homology class in $\widehat\HF(-Y,\mathfrak t_\xi)\otimes(\F_{(-1)}\oplus\F_{(0)})^{\otimes n-1}$. In addition, note that the Ozsv\'ath-Szab\'o contact invariant $\widehat c(Y,\xi)$, introduced in \cite{OSz.c}, is a homology class in the group $\widehat\HF(-Y,\mathfrak t_\xi)$. 
\begin{lemma}[Cavallo]
 \label{lemma:tower}
 With the assumptions made above, the Legendrian invariant $\Leg(\mathcal L)$ always belongs to $F^{-1}_*(\widehat c(Y,\xi)\otimes\textbf e_{1-n})$ in $c\HFL^-(-Y,L,\mathfrak t_\xi)$. Furthermore, the map $F$ sends an element with bigrading $(M,A)$ into one of grading $M-2A$, which implies that 
\begin{equation}
 \label{eq:d3}
 M(x(\mathcal L))=M(\Leg(\mathcal L))=-d_3(Y,\xi)+2A(x(\mathcal L))+1-n\:.
\end{equation}
\end{lemma}
In general, the invariant $\Leg(\mathcal L)$ can be vanishing, but the previous lemma shows that whenever $\widehat c(Y,\xi)$ is non-zero then $\Leg(\mathcal L)$ is actually non-torsion and belongs to the tower $F^{-1}_*(\widehat c(Y,\xi)\otimes\textbf e_{1-n})$. In particular, we know that this is the case when $(Y,\xi)$ is strongly symplectically fillable, see \cite{OSz.c}.

We can then state the following result from the second author; the proof appears in \cite{M.s} for knots, but it is consistent with the definition of $c\CFL^+(-Y,L,\mathfrak t_\xi)$, its homology $c\HFL^+(-Y,L,\mathfrak t_\xi)$ and $\Leg^+(\mathcal L)$ given before and thus also works for links.
\begin{lemma}[Matkovi\v c]
 \label{lemma:plus}
 The invariant $\Leg^+(\mathcal L)\in c\HFL^+(-Y,L,\mathfrak t_\xi)$ is non-zero if and only if both $\Leghat(\mathcal L)$ and $\widehat c(Y,\xi)$ are non-zero.
\end{lemma}

We now recall the definition of $\tau_\xi(L)$ from \cite{Hedden_xi}: assume that the contact structure $\xi$ on $Y$ is such that $\widehat c(Y,\xi)\neq[0]$, then we say that $\tau_\xi(L)$ is equal to the Alexander grading of the generator of the tower $\mathcal T$ in $c\HFL^-(-Y,L,\mathfrak t_\xi)$ such that $F_*(\mathcal T)=\widehat c(Y,\xi)\otimes \textbf e_{1-n}$; in other words, if we write $\mathcal T\cong\F[U]_{(d,s)}$ then $\tau_\xi(L)=s$. Such an integer is a smooth link invariant, but requires the choice of a specific contact structure on $Y$.

The Alexander grading of the cycle $x(\mathcal L)$ has been computed in \cite{O-S,Cavallo2} from the classical invariants of $\mathcal L$: 
\begin{equation}
 A(x(\mathcal L))=A(\Leg(\mathcal L))=\dfrac{\tb(\mathcal L)-\rot(\mathcal L)+n}{2}\:.
 \label{eq:Alexander}
\end{equation}

Equations \ref{eq:d3} and \ref{eq:Alexander} will be used to determine the 3-dimensional invariants (for instance, for Proposition \ref{prop:taut}). However, in Theorem \ref{thm:main} we also need the following observation.
\begin{prop}
 \label{prop:plus}
 If an $n$-component Legendrian link $\mathcal L$ in $(Y,\xi)$ with $\widehat c(Y,\xi)\neq[0]$ fulfills the equality $\tb(\mathcal L)-\rot(\mathcal L)=2\tau_\xi(L)-n$, then $\Leg^+(\mathcal L)$ is non-zero.
\end{prop}
\begin{proof}
  We have the condition that $\widehat c(Y,\xi)\neq[0]$; and then $\Leg(\mathcal L)$ is non-torsion. By assumption we can also write \[\dfrac{\tb(\mathcal L)-\rot(\mathcal L)+n}{2}=\tau_\xi(L)\] and we obtain from Equation \eqref{eq:Alexander} that $A(\Leg(\mathcal L))=\tau_\xi(L)$. Lemma \ref{lemma:tower} then implies that $\Leg(\mathcal L)$ is the generator of its tower in $c\HFL^-(-Y,L,\mathfrak t_\xi)$; in fact, since multiplication by $U$ drops the Alexander grading by 1, only a generator of a tower can have maximal Alexander grading. This means that $\Leg(\mathcal L)$ is not in the image of the $U$-action on $c\HFL^-(-Y,L,\mathfrak t_\xi)$ and then it cannot be in the kernel of $i_*^-$; which leads to $\widehat{\Leg}(\mathcal L)\neq[0]$. We conclude by applying Lemma \ref{lemma:plus}.
\end{proof}

\section{The invariant \texorpdfstring{$\Leghat$}{L} and Giroux torsion}
\label{section:three}

\subsection{Sutured Legendrian invariants}\label{ss:sfh}
We briefly review sutured reinterpretation of the knot Floer invariants, due to Stipsicz and V\'ertesi \cite{SV} for $\HFKhat$, Golla \cite{G.i} for $\HFK^-$, and Etnyre, Vela-Vick and Zarev \cite{EVVZ} for $\HFK^+$. 
Such interpretations can be naturally generalized to links and, for this reason, we write the results already in the setting we need.
\begin{convention}
 We parametrize every boundary component of the (null-homologous) Legendrian link complement by the meridian of the corresponding knot taking the slope $\infty$, and the Seifert framing taking the slope $0$. Thus, the boundary slope (of the dividing curves) equals the Thurston-Bennequin number, $\tb(\mathcal L_i)$ along the $i$-th component of $\mathcal L$.
\end{convention}

\subsubsection*{$\EH$-invariants and gluing maps} Honda, Kazez and Mati\'c \cite{HKM.eh} define $\EH(Y,\xi)$ of a contact manifold with convex boundary as a class in $\SFH(-Y,-\Gamma_\xi)$ where $\Gamma_\xi$ consists of dividing curves of $\xi$ on $\partial Y$. The contact invariant $\widehat c(Y,\xi)$ of a closed 3-manifold $Y$ is then identified with $\EH((Y,\xi)\backslash(B^3,\xi_\text{std}))$ in $\SFH(-Y(1),\mathfrak t_\xi)$, and in case of a Legendrian link $\mathcal L$ with link type $L$, the invariant $\EH(\mathcal L)$ is set to be $\EH(Y(L),\xi_{\mathcal L})$ in $\SFH(-(Y\backslash\nu L), -\Gamma_{\tb(\mathcal L)},\mathfrak t_\xi)$ where $\Gamma_{\tb(\mathcal L)}$ is a collection of pairs of oppositely oriented closed curves of slope $\tb(\mathcal L_i)$ on the boundary torus corresponding to the $i$-th component of $\mathcal L$. The crucial property of these sutured contact invariants is their behavior under gluing diffeomorphisms \cite{HKM.glue}: for sutured manifolds $(Y,\Gamma)\subset(Y',\Gamma')$, a contact structure $\zeta$ on $Y'\backslash Y$, compatible with $\Gamma\cup\Gamma'$, induces a map 
\[\Phi_\zeta: \SFH(-Y,-\Gamma) \rightarrow \SFH(-Y',-\Gamma')\]
which in case of contact manifolds with convex boundary connects the $\EH$-classes, that is
\[\Phi_\zeta(\EH(Y,\xi_Y))=\EH(Y',\xi_Y\cup\zeta).\]

Studying Legendrian links, the sutured manifolds, we are specifically interested in, are the link complements with various boundary slopes $(Y(L),\Gamma_{s})$, and the key gluing maps are
\[\sigma^\pm_{s,s'}: \SFH(-Y(L),-\Gamma_{s})\rightarrow \SFH(-Y(L),-\Gamma_{s'})\:,\] where $s=(s_1,...,s_i,...,s_n)$ and $s'=(s_1,...,s_i',...,s_n)$, associated to the addition of a basic slice along the $i$-th component, see \cite{Ho.I}.

\subsubsection*{Invariant $\Leghat$} Stipsicz and V\'ertesi \cite{SV} interpret the Legendrian invariant $\Leghat(\mathcal L)\in\widehat\HFL(-Y,L,\mathfrak t_\xi)$, 
the homology of $\widehat\CFL(-Y,L,\mathfrak t_\xi)$, as the $\EH$-invariant of the complement of a Legendrian link $(Y(L),\xi_{\mathcal L},\Gamma_{\tb(\mathcal L)})$ completed along all components by the negative basic slices with boundary slopes $\tb(\mathcal L_i)$ and $\infty$. Therefore, if we denote the completion $\xi_{\mathcal L}\cup\zeta^-_{\tb(\mathcal L),\infty}$ by $\overline{\xi_\infty}$, we have
\[\Leghat(\mathcal L)=\EH(Y(L),\overline{\xi_\infty}) \in\SFH(-Y(L),-\Gamma_\infty,\mathfrak t_\xi)\:;\] hence, we have a map \[\widehat i_*:\SFH(-Y(L), -\Gamma_{\tb(\mathcal L)},\mathfrak t_\xi)\longrightarrow\widehat\HFL(-Y,L,\mathfrak t_\xi)\cong\SFH(-Y(L),-\Gamma_\infty,\mathfrak t_\xi)\] which sends $\EH(\mathcal L)$ into $\widehat\Leg(\mathcal L)$, implying that if the latter is non-vanishing then also $\EH(\mathcal L)$ is.

\subsubsection*{Invariant $\Leg$} Golla \cite{G.i}, explicitly for knots in the $3$-sphere, and later Etnyre, Vela-Vick and Zarev \cite{EVVZ} give corresponding reinterpretation for the Legendrian invariant $\Leg(\mathcal L)\in c\HFL^-(-Y,L,\mathfrak t_\xi)$. For links we have that $\SFH(-Y(L),\Gamma_{k})$ with $k=(k_1,...,k_n)$ and $k_i\leq\tb(\mathcal L_i)$, together with 
\[\phi^-_{k,j}:\SFH(-Y(L),-\Gamma_{k}) \longrightarrow \SFH(-Y(L),-\Gamma_{j})\hspace{0.5cm}\text{ such that }\hspace{0.5cm}j_i\leq k_i,\]
the compositions of (appropriate) negative basic slice maps $\sigma^-_{s,s'}$ defined above, form a direct system \[\left(\{\SFH(-Y(L),\Gamma_{k})\}_{k\leq\tb(\mathcal L)},\: \{\phi^-_{k,j}\}_{|j|<|k|}\right)\] where $|k|=k_1+...+k_n$ and $|j|=j_1+...+j_n$.
Furthermore, their work shows that its direct limit $\underrightarrow{\SFH}(-Y,L,\mathfrak t_\xi)$ is mapped to $c\HFL^-(-Y,L,\mathfrak t_\xi)$, with $U$-action given by (any of) the maps $\sigma^+_{s,s'}$. One can then observe that the $\EH$-invariants of the negative stabilizations $\mathcal L^{i-}$ of $\mathcal L$ respect this direct system, so that the class $\EHdirect(\mathcal L)$ of the vector $\left(\EH(\mathcal L^{i_1-,...,i_n-})\right)_{|i|\in\N}$ is taken to $\Leg(\mathcal L)$. 

This means that the non-vanishing of $\Leg(\mathcal L)$ implies the non-vanishing of the $\EH$-invariant of $\mathcal L$ and all its negative stabilizations (for knots it is shown that the converse is also true).

\subsubsection*{Invariant $\Leg^+$} From the work of Etnyre, Vela-Vick and Zarev \cite{EVVZ} we can also consider a parallel inverse system with boundary slopes $k_i\geq\tb(\mathcal L_i)$: \[\left(\{\SFH(-Y(L),\Gamma_{k})\}_{k\geq\tb(\mathcal L)}, \{\phi^-_{j,k}\}_{|j|>|k|}\right)\:,\] whose inverse limit $\underleftarrow{\SFH}(-Y,L,\mathfrak t_\xi)$ is mapped to $c\HFL^+(-Y,L,\mathfrak t_\xi)$. They define the inverse limit invariant $\EHinverse(\mathcal L)$ to be the class of the vector $\left(\EH(Y(L),\overline{\xi_{k}})\right)_{|k|\geq\tb(\mathcal L)}$, where $\overline{\xi_k}$ equals the extension by $2n$ negative basic slices $\xi_L\cup\zeta^-_{\tb(\mathcal L),\infty}\cup\zeta^-_{\infty,k}$. Under the above map, the invariant $\EHinverse(\mathcal L)$ is sent to $\Leg^+(\mathcal L)$.

\subsection{Giroux torsion}
\label{subsection:Gtor}
We will refer to a thickened torus contactomorphic to \[\big(T^2\times[0,1],\:\eta_\pi=\Ker(\sin(\pi t)\mbox{d} x+\cos(\pi t)\mbox{d} y)\big)\] as a \emph{$\pi$-torsion layer}. In fact, a contact manifold $(Y, \xi)$ is said to have \emph{Giroux $k\pi$-torsion} along $T$ if there exists a $\pi$-torsion layer which embeds into $(Y, \xi)$, so that each $T^2\times\{t\}$ is isotopic to $T$ and $\eta_{k\pi}$ is obtained by stacking $k$ copies of $\eta_\pi$. The invariant was introduced by Giroux in \cite{Giroux}; it is a custom to refer to the $2\pi$-torsion as  Giroux torsion, and then to $\pi$-torsion as half Giroux torsion.

According to the classification of tight contact structures on a thickened torus, due to Honda, there exist exactly two tight contact structures on $T^2\times I$ with a pair of dividing curves of the same slope $s$ at both boundary tori and with $I$-twisting equal to $\pi$ \cite[Theorem 2.2 (3)]{Ho.I}; they are universally tight \cite[Proposition 2.1]{HKM.tor}. In particular, when we split the $\pi$-torsion layer into basic slices, they all have the same sign, determined by the sign of the embedded basic slice from the boundary slope $s\in\Z$ to $\infty$.

Finally, we notice that if we insert a $\pi$-torsion layer along the boundary $T$ of a tubular neighborhood of a Legendrian knot $K$, then the same result can be obtained by a \emph{Lutz twist} on the transverse push-off of $K$; we say that we are adding half Giroux torsion along a knot. When the knot is null-homologous, this operation does not change the $\spinc$-structure 
and it shifts the $d_3$-invariant by $-\slk(K)$; meanwhile, for the Legendrian knot it preserves the Thurston-Bennequin number and reverses the orientation and the self-linking number of its transverse push-off. See \cite{DGS} for details. While most of this obviously extends to links, we present below a computation for the shift of the $d_3$-invariant when adding half Giroux torsion along a link.

In \cite{DGS} Ding, Geiges and Stipsicz constructed a surgery presentation of the contact manifold obtained by adding half Giroux torsion along a knot in $(Y,\xi)$. 
\begin{figure}[ht]
 \centering
 \def\svgwidth{5cm}
\begingroup%
  \makeatletter%
  \providecommand\color[2][]{%
    \errmessage{(Inkscape) Color is used for the text in Inkscape, but the package 'color.sty' is not loaded}%
    \renewcommand\color[2][]{}%
  }%
  \providecommand\transparent[1]{%
    \errmessage{(Inkscape) Transparency is used (non-zero) for the text in Inkscape, but the package 'transparent.sty' is not loaded}%
    \renewcommand\transparent[1]{}%
  }%
  \providecommand\rotatebox[2]{#2}%
  \newcommand*\fsize{\dimexpr\f@size pt\relax}%
  \newcommand*\lineheight[1]{\fontsize{\fsize}{#1\fsize}\selectfont}%
  \ifx\svgwidth\undefined%
    \setlength{\unitlength}{2421.74097666bp}%
    \ifx\svgscale\undefined%
      \relax%
    \else%
      \setlength{\unitlength}{\unitlength * \real{\svgscale}}%
    \fi%
  \else%
    \setlength{\unitlength}{\svgwidth}%
  \fi%
  \global\let\svgwidth\undefined%
  \global\let\svgscale\undefined%
  \makeatother%
  \begin{picture}(1,0.39866491)%
    \lineheight{1}%
    \setlength\tabcolsep{0pt}%
    \put(0,0){\includegraphics[width=\unitlength,page=1]{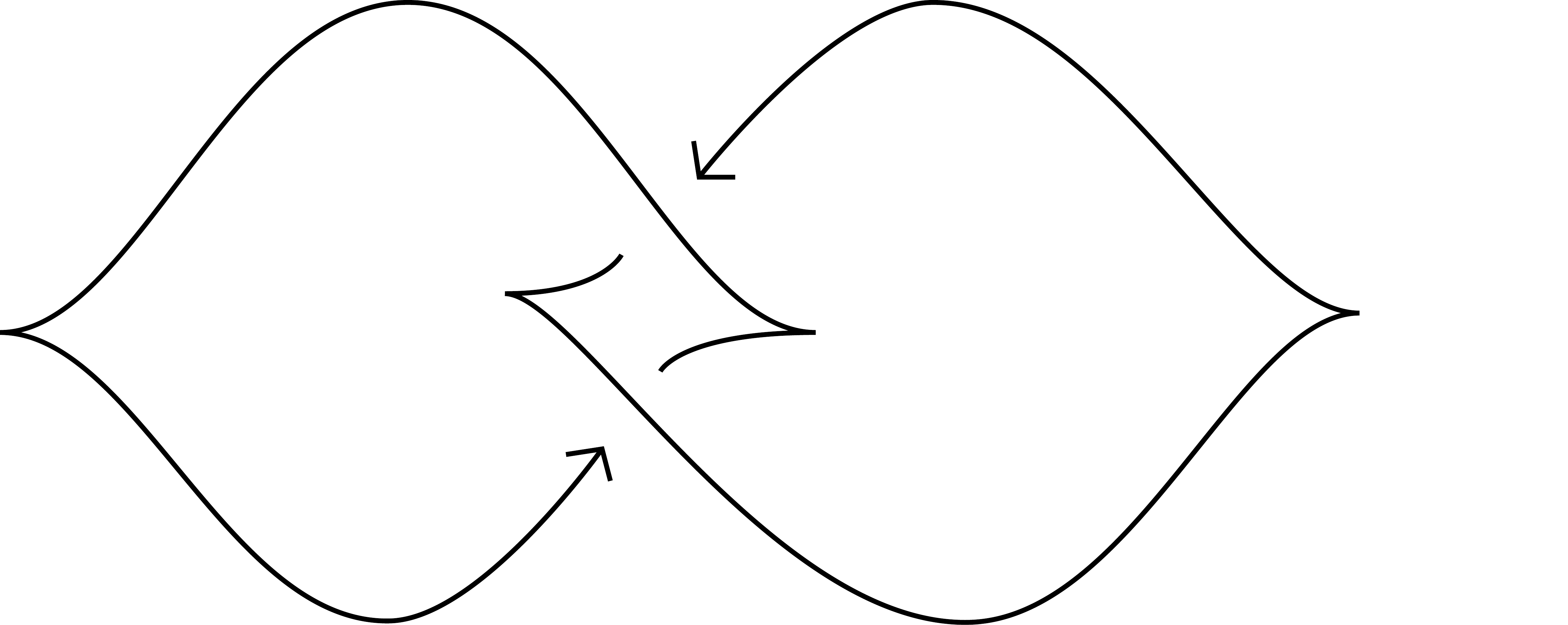}}%
    \put(0.78308478,0.05998923){\color[rgb]{0,0,0}\makebox(0,0)[lt]{\lineheight{80}\smash{\begin{tabular}[t]{l}$\mathcal L$\end{tabular}}}}%
  \end{picture}%
\endgroup%

 \hspace{2cm}
 \def\svgwidth{8cm}
\begingroup%
  \makeatletter%
  \providecommand\color[2][]{%
    \errmessage{(Inkscape) Color is used for the text in Inkscape, but the package 'color.sty' is not loaded}%
    \renewcommand\color[2][]{}%
  }%
  \providecommand\transparent[1]{%
    \errmessage{(Inkscape) Transparency is used (non-zero) for the text in Inkscape, but the package 'transparent.sty' is not loaded}%
    \renewcommand\transparent[1]{}%
  }%
  \providecommand\rotatebox[2]{#2}%
  \newcommand*\fsize{\dimexpr\f@size pt\relax}%
  \newcommand*\lineheight[1]{\fontsize{\fsize}{#1\fsize}\selectfont}%
  \ifx\svgwidth\undefined%
    \setlength{\unitlength}{3234.45885247bp}%
    \ifx\svgscale\undefined%
      \relax%
    \else%
      \setlength{\unitlength}{\unitlength * \real{\svgscale}}%
    \fi%
  \else%
    \setlength{\unitlength}{\svgwidth}%
  \fi%
  \global\let\svgwidth\undefined%
  \global\let\svgscale\undefined%
  \makeatother%
  \begin{picture}(1,0.39789619)%
    \lineheight{1}%
    \setlength\tabcolsep{0pt}%
    \put(0,0){\includegraphics[width=\unitlength,page=1]{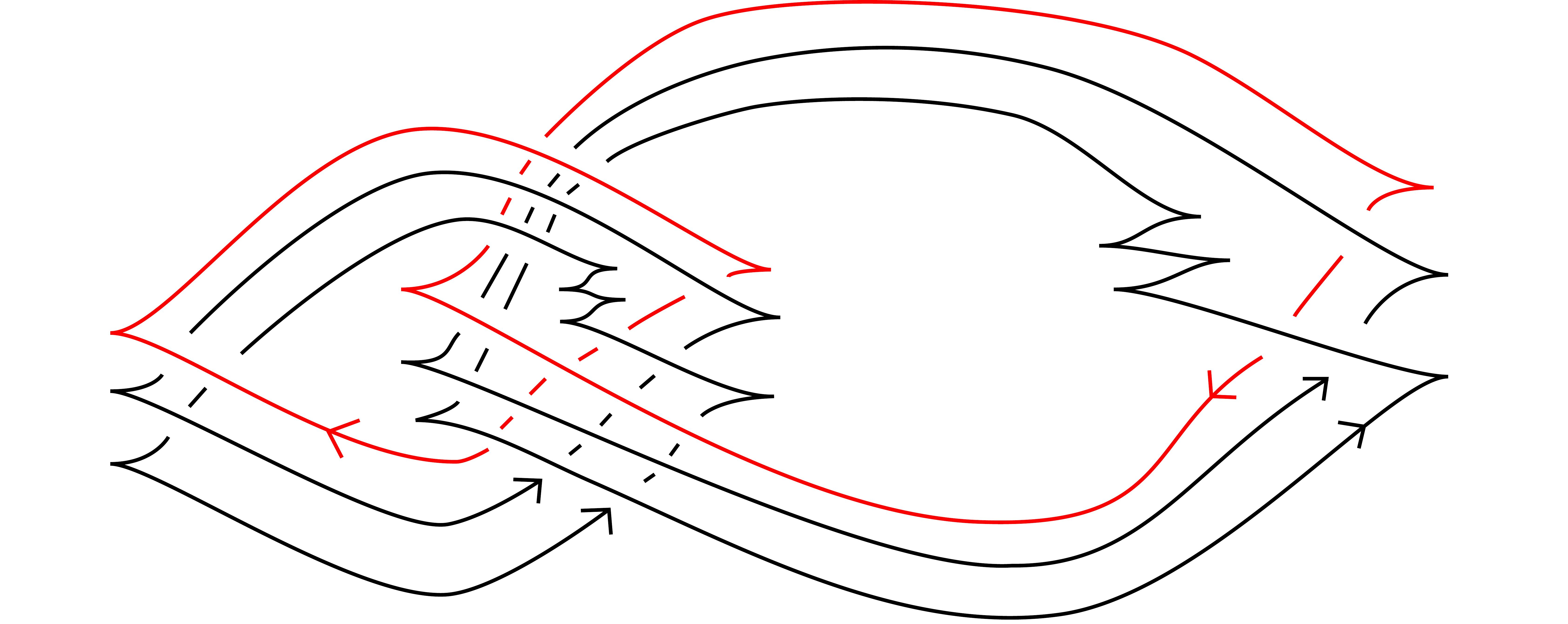}}%
    \put(0.17901453,0.32595429){\color[rgb]{0,0,0}\makebox(0,0)[lt]{\lineheight{80}\smash{\begin{tabular}[t]{l}$\mathcal L'$\end{tabular}}}}%
    \put(-0.00052531,0.08082609){\color[rgb]{0,0,0}\makebox(0,0)[lt]{\lineheight{80}\smash{\begin{tabular}[t]{l}$+1$\end{tabular}}}}%
    \put(-0.00187193,0.12914004){\color[rgb]{0,0,0}\makebox(0,0)[lt]{\lineheight{80}\smash{\begin{tabular}[t]{l}$+1$\end{tabular}}}}%
    \put(0.93425291,0.20599104){\color[rgb]{0,0,0}\makebox(0,0)[lt]{\lineheight{80}\smash{\begin{tabular}[t]{l}$+1$\end{tabular}}}}%
    \put(0.93359038,0.14040268){\color[rgb]{0,0,0}\makebox(0,0)[lt]{\lineheight{80}\smash{\begin{tabular}[t]{l}$+1$\end{tabular}}}}%
  \end{picture}%
\endgroup%

 \caption{A Legendrian link $\mathcal L$ (left) and the link $\mathcal L'$ (in red) after adding $\pi$-torsion along the components of $\mathcal L$ (right). The non-loose link $\mathcal M$, constructed in Theorem \ref{thm:plus}, is a negative destabilization of $\mathcal L'$.}
 \label{Hopf}
\end{figure}
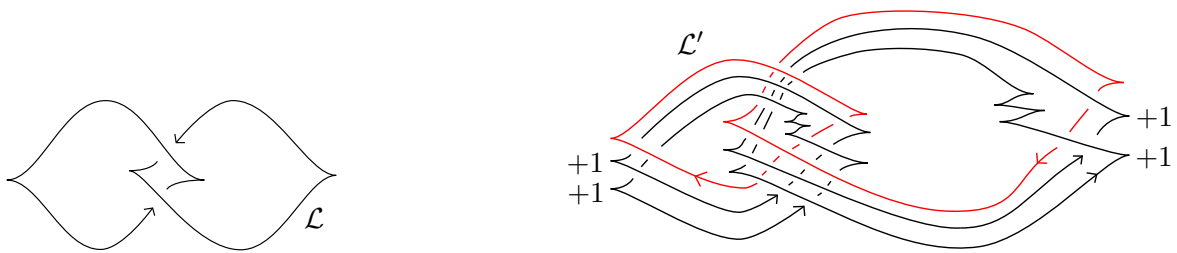
In Figure \ref{Hopf}, we show this construction repeated $n$-times along the transverse push-off of the Legendrian link $\mathcal L$: we use this presentation to compute the change of the $d_3$-invariant.

\begin{prop}
 \label{prop:Hopf}
 Assume that $\xi'$ is the contact structure on $Y$ given by adding half Giroux torsion along (all components of) the Legendrian link $\mathcal L\hookrightarrow(Y,\xi)$; then $d_3(\xi')=d_3(\xi)-\tb(\mathcal L)+\rot(\mathcal L)$.
\end{prop}
\begin{proof}
 We proceed by induction on $n$, the number of components of $L$ which is the link type of $\mathcal L$.
 Suppose that $n=1$; since $(Y,\xi')$ is obtained from $(Y,\xi)$ by a Lutz twist along the transverse push-off $\mathcal T$ of $\mathcal L$, the Hopf invariants of the two contact structures differ by its self-linking number, as shown in \cite{DGS}. Therefore, we are done because $\slk(\mathcal T)=\tb(\mathcal L)-\rot(\mathcal L)$.

 Assume now that the result is true for $(n-1)$-component links; let us denote by $\widehat{\mathcal L}$ the link $\mathcal L$ without its last component $\mathcal L_n$, and by $\widehat\xi$ the contact structure obtained after applying a Lutz twist along the transverse push-off of $\widehat{\mathcal L}$.
 The previous result for knots implies that \[d_3(Y,\xi')=d_3(Y,\widehat\xi)-\tb(\mathcal L''_n)+\rot(\mathcal L''_n)\:,\] where $\mathcal L''_n$ is the knot $\mathcal L_n$ seen inside $(Y,\widehat\xi)$. By the inductive step we obtain \[d_3(Y,\xi')=d_3(Y,\xi)-\tb(\widehat{\mathcal L})-\tb(\mathcal L''_n)+\rot(\widehat{\mathcal L})+\rot(\mathcal L''_n)\:;\] hence, we only need to compute the invariants of $\mathcal L''_n$ using the formulae in \cite{LOSSz}:\[\tb(\mathcal L''_n)=\tb(\mathcal L_n)+\dfrac{\det\left(M(0)\right)}{\det M}\hspace{0.5cm}\text{ and }\hspace{0.5cm}\rot(\mathcal L''_n)=\rot(\mathcal L_n)-\left\langle\begin{pmatrix}
  \ell_{1,n}\\
  \ell_{1,n}\\
  \vdots\\
  \ell_{n-1,n}\\
  \ell_{n-1,n}
\end{pmatrix},\:M^{-1}\cdot\begin{pmatrix}
  r_1\\
  r_1-2\\
  \vdots\\
  r_{n-1}\\
  r_{n-1}-2
\end{pmatrix}\right\rangle\:,\] where $\ell_{i,j}:=\lk(L_i,L_j)\:,t_i:=\tb(\mathcal L_i)\:,r_i:=\rot(\mathcal L_i)$ and $M$ and $M(0)$ are the following matrices \[M=\left(\begin{array}{@{}c|c@{}|c|c@{}}
  \begin{matrix}
  t_1+1 & t_1  \\
  t_1 & t_1-1
  \end{matrix}
  & 
  \begin{matrix}
  \ell_{1,2} & \ell_{1,2}  \\
  \ell_{1,2} & \ell_{1,2}
  \end{matrix}
  & \vdots & \vdots \\
\hline
   \begin{matrix}
  \ell_{1,2} & \ell_{1,2}  \\
  \ell_{1,2} & \ell_{1,2}
  \end{matrix}
   &
  \begin{matrix}
   t_2+1 & t_2  \\
   t_2 & t_2-1
  \end{matrix}
  & \vdots & \vdots \\ \hline \cdots & \cdots & \ddots & \vdots \\ \hline
  \cdots & \cdots & \cdots &
  \begin{matrix}
  t_{n-1}+1 & t_{n-1}  \\
  t_{n-1} & t_{n-1}-1
  \end{matrix}
\end{array}\right)\] and \[M(0)=\left(\begin{array}{@{}c|c@{}}
  M
  & \begin{matrix}
  \ell_{1,n}\\
  \ell_{1,n}\\
  \vdots\\
  \ell_{n-1,n}\\
  \ell_{n-1,n}
 \end{matrix} \\
 \hline
  \ell_{1,n}\:\:\ell_{1,n}\cdots\ell_{n-1,n}\:\:\ell_{n-1,n} & 0
 \end{array}\right)\:.\]
 We start by observing that \[\faktor{\Z}{|\det(M(0))|\cdot\Z}\cong H_1(Y_{(t_1+1,...,t_{n-1}-1,0)}(L);\Z)\cong H_1(Y_0(L_n);\Z)\cong\Z\] because $L$ is null-homologous, which more specifically means that $L_n$ is null-homologous, and the surgery on $\widehat L$ does not change the link type. Therefore, we get $\det(M(0))=0$ and then $\tb(\mathcal L''_n)=\tb(\mathcal L_n)$. We continue with the following computation: \[M\cdot\left[\begin{pmatrix}
  r_1-2t_1\\
  2-r_1+2t_1\\
  \vdots\\
  r_{n-1}-2t_{n-1}\\
  2-r_{n-1}+2t_{n-1}
 \end{pmatrix}-2\begin{pmatrix}
  \lk(L_n,\widehat L)\\
  -\lk(L_n,\widehat L)\\
  \vdots\\
  \lk(L_n,\widehat L)\\
  -\lk(L_n,\widehat L)
 \end{pmatrix}\right]=\begin{pmatrix}
  r_1\\
  r_1-2\\
  \vdots\\
  r_{n-1}\\
  r_{n-1}-2
 \end{pmatrix}\] and substituting in the formula for the rotation number ($\det(M)\neq 0$ since $Y$ is a rational homology sphere) gives \[\rot(\mathcal L''_n)=\rot(\mathcal L_n)-\left(\ell_{1,n}\:\:\ell_{1,n}\cdots\ell_{n-1,n}\:\:\ell_{n-1,n}\right)\cdot\begin{pmatrix}
  r_1-2t_1\\
  2-r_1+2t_1\\
  \vdots\\
  r_{n-1}-2t_{n-1}\\
  2-r_{n-1}+2t_{n-1}
 \end{pmatrix}+\] \[+2\left(\ell_{1,n}\:\:\ell_{1,n}\cdots\ell_{n-1,n}\:\:\ell_{n-1,n}\right)\cdot\begin{pmatrix}
  \lk(L_n,\widehat L)\\
  -\lk(L_n,\widehat L)\\
  \vdots\\
  \lk(L_n,\widehat L)\\
  -\lk(L_n,\widehat L)
 \end{pmatrix}=\rot(\mathcal L_n)-2\lk(L_n,\widehat L)\:.\]
 We conclude that \[d_3(Y,\xi')=d_3(Y,\xi)-\tb(\widehat{\mathcal L})-\tb(\mathcal L_n)+\rot(\widehat{\mathcal L})+\rot(\mathcal L_n)-2\lk(L_n,\widehat L)=\] \[=d_3(Y,\xi)-\tb(\widehat{\mathcal L})-\tb_n(\mathcal L)+\rot(\mathcal L)=d_3(Y,\xi)-\tb(\mathcal L)+\rot(\mathcal L)\:.\]
\end{proof}

\subsection{\texorpdfstring{$\EH$}{EH}-invariant and taut foliations}\label{ss:taut}
From the work of Eliashberg and Thurston \cite{ET}, it is well known that the contact structures on a closed 3-manifold, which arise as perturbations of taut foliations, are always symplectically semi-fillable and universally tight. Therefore, by Ozsv\'ath and Szab\'o \cite{OSz}, their contact invariant $c(Y,\xi)\in\HFhat(Y; \F[H_2(Y;\Z)])$ is non-zero. 

Furthermore, it was proved by Etnyre and Vela-Vick \cite{VV,EVV} that the $\EH$-invariant (in fact, $\Leghat$) of the contact perturbation of the open book foliation in the complement of the binding is also non-zero.

The relationship between contact invariants and sutured manifold hierarchies is the central focus in \cite[Section 6]{HKM.eh} of Honda, Kazez and Mati\'c, from where the following general statement can be extracted.

\begin{thm}\label{thm:EH(taut)}
If a contact structure $\xi$ on an (irreducible) sutured manifold $(Y,\Gamma)$ arises from a taut sutured manifold hierarchy, as a contact perturbation of the corresponding taut foliation, then it has non-zero $\EH$-invariant. 
\end{thm}

\begin{proof}
Following Gabai \cite{Gabai}, such a taut foliation is constructed through a well-groomed sutured manifold hierarchy, which ends in a product.  

A contact structure which is a perturbation of the product foliation has non-zero $\EH$-invariant. Indeed, we can use the same argument as Etnyre and Vela-Vick in \cite{EVV}: a product sutured manifold $(\Sigma\times I, \partial\Sigma)$ admits exactly one tight contact structure. It is the same contact manifold as the handlebody used in Heegaard splitting associated to open book decompositions. Hence (as we can always complete the manifold with positive monodromy), its $\EH$-invariant is non-zero.

In addition, Honda, Kazez and Mati\'c \cite{HKM.tor, HKM.eh} associated to well-groomed sutured decomposition a convex decomposition (with cutting surfaces having boundary parallel dividing sets), and showed that gluing maps are injective in this case.

This finishes the proof, as the $\EH$-invariant of the product is through hierarchy injectively mapped to $\EH(\xi)$.
\end{proof}

Theorem \ref{thm:taut} and Proposition \ref{prop:taut}  are then proved as follows.

\begin{proof}[Proof of Theorem \ref{thm:taut}]
According to Theorem \ref{thm:EH(taut)}, the invariant $\EH(Y(L),\Gamma_\infty,\xi_L)$ is non-zero.  
Now, if we isotope the boundary torus over the bypass disk (defined by a boundary parallel dividing curve on the Seifert surface), the trace of this isotopy is a basic slice with slopes $\infty$ and $0$. Hence, we have an embedded Legendrian knot complement $(Y(L),\Gamma_0)=\mathcal L$, and the described trace of isotopy is the completing basic slice; therefore $\Leghat(\mathcal L)=\EH(Y(L),\Gamma_\infty,\xi_L)$.

Since the (Thurston norm minimizing) Seifert surface is a leaf of the taut foliation which perturbs into $\xi_L$, it contain no dividing curve in $(Y(L),\Gamma_0)$ and we have $\tb(\mathcal L)-\rot(\mathcal L)=||L||_T$; the Alexander grading is then computed by Equation \ref{eq:Alexander}.
\end{proof}

\begin{proof}[Proof of Proposition \ref{prop:taut}]
Assumptions ensure that $(Y(L), \Gamma_\infty)$ is a taut sutured manifold and hence there exists $\mathcal L$ as in Theorem \ref{thm:taut}. 
Starting from the fact that \[||L||_T=\tb(\mathcal L)-\rot(\mathcal L)=2A(\widehat{\mathfrak L}(\mathcal L))-n\:,\] we obtain by Equation \ref{eq:d3} that $\mathcal L$ lives in \[\begin{aligned}d_3(\xi_0)=2A(\widehat{\mathfrak L}(\mathcal L))-M(\widehat{\mathfrak L}(\mathcal L))+1-n=1+(2A(\widehat{\mathfrak L}(\mathcal L))-n)-&M(\widehat{\mathfrak L}(\mathcal L))= \\ 
&=1+||L||_T-M(\widehat{\mathfrak L}(\mathcal L))\:.\end{aligned}\]

As for \cite[Theorem 1.8]{E} the two families of non-loose links are obtained by adding boundary parallel $k\pi$-torsion. For even $k$ the 3-dimensional invariant of the ambient equals $d_3(\xi)=d_3(\xi_0)$, for odd $k$ we get from Proposition \ref{prop:Hopf} that $d_3(\xi')=d_3(\xi_0)-||L||_T$.
\end{proof}
\begin{proof}[Proof of Corollary \ref{cor:d3}]
 We only write the proof for $d_3(\xi)$.
 Since the link has thickness equal to $\text{th}(L)$, we obtain that $\widehat\HFL(L)$ is supported in $\text{th}(L)+1$ lines in the $(M,\:A)$-plane such that $A=M+k$ and $|k-\tau(L)|\leq\text{th}(L)$. 
 This in particular holds for $\widehat{\mathfrak L}(\mathcal L)\in\HFLhat(L^*)$ and, since $\tau(L^*)=-\tau^*(L)$ from \cite{Cavallo}, one has \[-M(\widehat{\mathfrak L}(\mathcal L))=-A(\widehat{\mathfrak L}(\mathcal L))+k=-\dfrac{n+||L||_T}{2}+k\leq-\dfrac{n+||L||_T}{2}-\tau^*(L)+\text{th}(L)\] and
 \[-M(\widehat{\mathfrak L}(\mathcal L))=-A(\widehat{\mathfrak L}(\mathcal L))+k=-\dfrac{n+||L||_T}{2}+k\geq-\dfrac{n+||L||_T}{2}-\tau^*(L)-\text{th}(L)\]
 which then imply the claim by substituting the values of $M(\widehat{\mathfrak L}(\mathcal L))$ in the previous formula.

 Now if we assume that $L$ is $\HFLhat$-thin then $\text{th}(L)=0$ and, according to \cite{Cavallo}, one has $\tau^*(L)=-\frac{n-1+\sigma(L)}{2}$. 
\end{proof}

\subsection{\texorpdfstring{$\Leg$}{Leg} versus \texorpdfstring{$\Leghat$}{Leg hat} versus \texorpdfstring{$\Leg^+$}{Leg +}}\label{ss:^vs-}
From the sutured interpretation of Legendrian invariants in link Floer homologies (Subsection \ref{ss:sfh}), it is clear that they are associated to different contact structures on the link complement. However, it is not immediately clear in what way the links with non-vanishing $\Leg^+$ or $\Leghat$ behave differently from the ones with non-vanishing only $\Leg$. For example, Legendrian surgery does not seem to distinguish between them \cite[Theorem 1.1]{M.s}.
One way in which knots with non-zero $\Leg^+$ are special is when considering positive contact surgeries: the resulting contact manifold will have non-zero invariant only if $\Leg^+$ of the surgered knot is non-zero \cite[Theorem 1.7]{M.s}. This does not suffice though, as observed by Golla in \cite[Theorem 1.1]{G.s}. Here, we notice another property of links with non-zero $\Leg^+$.

\begin{thm}\label{thm:plus}
If a link $L\hookrightarrow Y$ admits a Legendrian realization $\mathcal L$ in $(Y,\xi)$ with non-zero $\Leg^+(\mathcal L)\in c\HFL^+(-Y,L,\mathfrak t_\xi)$, then there exists another Legendrian representative $\mathcal M$ of $L$, in an overtwisted structure on $Y$, such that $\EH(\mathcal M)\neq[0]$. In particular, the link $\mathcal M$ is non-loose.
\end{thm}
\begin{proof}
Take the standard Legendrian link complement for $\mathcal L$ in $(Y,\xi)$ and extend it along all boundary components, by a $\pi$-torsion layer of (both) boundary slopes $\tb(\mathcal L_i)$ along the $i$-th component of $\mathcal L$. If we close the resulting link complement, name it $\mathcal L'$, by the unique solid tori with meridional slope $\infty$ and boundary slope $\tb(\mathcal L_i)$, we get an overtwisted structure $\xi'$ on $Y$. Indeed, smoothly we have not changed the manifold and the resulting structure is overtwisted because in (each) $\pi$-torsion layer there is a convex torus with dividing curves, and hence Legendrian divides, of slope $\infty$, the latter bounding overtwisted disk in the Dehn filling. 
In fact, the manifold $(Y,\xi')$ is obtained by a Lutz twist along the transverse push-off of $\mathcal L$.

On the other hand, the boundary parallel tori in the $\pi$-torsion layer take all possible slopes, and the signs of all basic slices in any of its decompositions agree; we choose them to be negative. 
In particular, we can then decompose the $\pi$-torsion layers in $(Y,\xi')\backslash \nu\mathcal L'$ into negative basic slices, along $\mathcal L_i$ with boundary slopes $\tb(\mathcal L_i),\infty, m_i, m_i-1, \cdots, \tb(\mathcal L_i)+1,\tb(\mathcal L_i)$:
$$\zeta^-_{\tb(\mathcal L_i),\infty}\cup\zeta^-_{\infty,m_i}\cup\zeta^-_{m_i,m_i-1}\cup\cdots\cup\zeta^-_{\tb(\mathcal L_i)+1,\tb(\mathcal L_i)},$$ where $m=(m_1,...,m_n)$ is such that for each $k$ satisfying $k_i\geq m_i$ the sutured invariant $\EH(Y(L),\overline{\xi_{k}})$ is non-zero; such $m$ exists because we assumed that $\Leg^+(\mathcal L)\neq [0]$. Then, the sutured manifold $(Y(L),\overline{\xi_m})$ is the complement of a Legendrian link, name it $\mathcal M$, which is tight (it has non-zero $\EH$-invariant) and which embeds into $(Y,\xi')\backslash \nu\mathcal L'$. 
Moreover, the two Legendrian links $\mathcal M$ and $\mathcal L'$ differ by the stacks of negative basic slices $\zeta^-_{m_i,m_i-1}\cup\zeta^-_{m_i-1,m_i-2}\cup\cdots\cup\zeta^-_{\tb(\mathcal L_i)+1,\tb(\mathcal L_i)}$; in other words, $\mathcal L'$ is obtained from $\mathcal M$ by applying $|m|-\tb(\mathcal L_1)-...-\tb(\mathcal L_n)$ negative stabilizations.

In summary, the Legendrian link $\mathcal M$ is clearly in the link type $L\hookrightarrow Y$. It is non-loose, because the $\EH$-invariant of its complement is non-zero, $\EH(\mathcal M)=\EH(Y(L),\overline{\xi_m})\neq [0]$, but it stabilizes to $\mathcal L'$ which has boundary parallel half Giroux torsion and thus closes to an overtwisted structure $(Y,\xi')$. 
\end{proof}

This result also concludes the proof of Theorem \ref{thm:main}.
\begin{proof}[Proof of Theorem \ref{thm:main}]
 Theorem follows by combining Propositions \ref{prop:plus} and \ref{prop:Hopf} and Theorem \ref{thm:plus}.
\end{proof}

\begin{rmk}\label{rmk:sub}
 Since a submanifold of a tight manifold is tight, we notice that along with the non-loose Legendrian link $\mathcal M$ we get a family of non-loose Legendrian links $\{\mathcal M_i\}_{i=1,...,2^{n}}$, built from $\mathcal L$ by attaching a pair of negative basic slices $\zeta^-_{\tb(\mathcal L_i),\infty}\cup\zeta^-_{\infty,m_i}$ only along some components. These realizations are not related by stabilizations to each other; in general, they will even live in different contact structures on $Y$ (each in the one where the Lutz twist is performed only along the chosen components).
\end{rmk}

Whenever $\Leg^+(\mathcal L)$ vanishes, any link $\mathcal M$ constructed as above will have vanishing $\EH(\mathcal M)$.
However, we expect that when $\Leghat(\mathcal L)$ is non-zero, the link $\mathcal M$ (for large enough $m$) is still non-loose (even if $\widehat c(Y,\xi)=[0]$). The case of taut foliations (from Subsection \ref{ss:taut}) is a particular such case where conjecture is confirmed.
\begin{conj}\label{conj:hat}
 If an $n$-component link $L\hookrightarrow Y$ admits a Legendrian realization $\mathcal L$ in $(Y,\xi)$ with non-zero $\Leghat(\mathcal L)$, then there exists a non-loose realization $\mathcal M$ of $L$ in an overtwisted structure $\eta$ on $Y$ with $\mathfrak t_\eta=\mathfrak t_\xi\in\spinc(Y)$ and $d_3(\eta)=d_3(\xi)-2A(\Leghat(\mathcal L))+n$.
\end{conj}

On the other hand, if the Legendrian link $\mathcal L$ is a positive stabilization, then the above constructed $\mathcal M$ is loose. Hence, we can ask ourselves about the converse of Conjecture \ref{conj:hat}; namely, whether the vanishing of $\Leghat(\mathcal L)$ (for knots with non-zero $\Leg(\mathcal L)$) implies that $\mathcal M$ is loose.
In support of a positive answer, we notice that both are true for negative torus knots, by comparison of \cite{M.n} and \cite{EMM}.

\begin{proof}[Proof of Proposition \ref{prop:T(p,-q)}]
 In \cite{EMM} it is shown that adding (half) Giroux torsion preserves tightness of a torus knot complement, if and only if it can be presented by a totally 2-inconsistent path. In the surgery presentation \cite[Figure 2]{M.n} 2-inconsistent corresponds to initial surgeries of the two singular fibers not being stabilized fully oppositely, and totally 2-inconsistent means they are both either fully positive or fully negative (see \cite[Definition 2.2]{M.n}). 

 Translating this back to the description of transverse strongly non-loose negative torus knots through large negative surgeries along them in \cite{M.n}, neither of the vertices being fully negative means that $U\cdot\Leg(\mathcal L)$ vanishes (as a single positive stabilization is loose), and both being fully positive means that $\Leg(\mathcal L)$ does not live in $U\cdot\HFK^-$ (since in this case the dual characteristic vector, corresponding to the surgery, is terminal in the full path \cite{M.t}).
\end{proof}

\begin{proof}[Proof of Corollary \ref{cor:T(p,-q)}]
If we write the knot Floer homology of a positive torus knot as 
\[ \HFK^-(T_{p,q})\cong \bigoplus_i \big(\F[U]/(U^{n_i})\big)_{(M_i,A_i)} \text{ where } n_0=\infty,\:n_i\in\N_{\geq1} \text{ for } i\geq1, \]
a generator of the summand $\F[U]/(U^{n_i})$ with $n_i>1$ cannot be represented by a Legendrian invariant of any $T_{p,-q}$ by the proof of Proposition \ref{prop:T(p,-q)}.

But the only torus knots with torsion order (as in \cite{JMZ}) equal to $1$ are $T_{2,-2n-1}$ for $n\geq1$, and for these the Legendrian invariants of the knots $L_{k,l}$ (in the notation of \cite{LOSSz}) present generators for all $U$-torsion elements of $\HFK^-(T_{2,2n+1})$ (as shown in \cite[Example 4.10]{M.n}).
\end{proof}

\section{Non-loose realizations}
\label{section:four}
\subsection{Quasi-positive links}\label{ss:qp}
A particular case when we can apply Theorem \ref{thm:main} is the one of quasi-positive links. We recall that in \cite{Hedden_quasi,Cavallo3} it is shown that for every $n$-component quasi-positive link $L$ in $S^3$ one has $\tau(L)=G_4(L)$, the \emph{big slice genus} of $L$ which is defined as \[G_4(L)=g(\Sigma)+n-k=\dfrac{n-\chi_4(L)}{2}\:,\] where $\Sigma$ is the properly embedded, compact, oriented surface in $D^4$ such that $\partial \Sigma=L$ and $\chi(\Sigma)$ is maximal, i.e. $\chi_4(L):=\max_\Sigma\{\chi(\Sigma)\}$; the integer $k$ is the number of connected components of $\Sigma$. Note that $G_4(L)=0$ if and only if the link $L$ is smoothly slice.
\begin{cor}\label{cor:qp}
 Every quasi-positive $n$-component link in $S^3$ admits a non-loose realization with non-zero $\EH$-invariant in the overtwisted structure $\eta$ such that $d_3(\eta)=\chi_4(L)$.
\end{cor}
 \begin{proof}
 It is proved in \cite{Hedden_quasi,Cavallo3} that for every Legendrian link $\mathcal L$ in $(S^3,\xi_{\text{std}})$, with link type $L$, we have the following version of the slice Thurston-Bennequin inequality:
 \begin{equation}
  \label{eq:TB}
  \tb(\mathcal L)-\rot(\mathcal L)\leq\tb(\mathcal L)+|\rot(\mathcal L)|\leq2\tau(L)-n\leq2G_4(L)-n=-\chi_4(L)\:.   
 \end{equation}
 If we represent $\mathcal L^T$, the transverse push-off of $\mathcal L$, as the closure of a quasi-positive braid then there is a surface $\Sigma$ as before such that \[-\chi(\Sigma)=\slk(\mathcal L^T)=\tb(\mathcal L)-\rot(\mathcal L)\] which means that all the inequalities in Equation \eqref{eq:TB} are actually equalities, and then \[\tb(\mathcal L)-\rot(\mathcal L)=2\tau(L)-n\:.\] 
 Therefore, we can conclude using our main theorem; moreover, one has \[d_3(\eta)=0-\tb(\mathcal L)+\rot(\mathcal L)=n-2G_4(L)=\chi_4(L)\:.\]
\end{proof}

In particular, with this corollary we confirm the existence of a non-loose realization with non-zero $\EH$-invariant for all the strongly quasi-positive links for which the perturbation of the taut foliation landed in the standard structure.

\subsection{Fibered links}
For strongly quasi-positive fibered links, the contact structure on the link complement from the open book fibration is exactly the one with non-vanishing $\Leg^+$-invariant. Hence, both arguments (relying on the taut foliation and on the sharpness of Thurston-Bennequin inequality) give us the same non-loose realizations in $d_3=-||L||_T$. Except in the case of the unknot, when the boundary parallel torus is compressible, which is covered only by the second argument; in fact, this is the only fibered link which has non-loose Legendrian realizations, but no transversely non-loose one. 
On the other hand, we have noticed that when the link does not satisfy the (classical) Thurston-Bennequin equality, the two arguments prove existence of different non-loose realizations. We illustrate this here in the case of a fibered, quasi-positive knot, which is not strongly quasi-positive.

Let us consider the knot $8_{20}$ in Figure \ref{8_20}: we have that $\tau(8_{20})=g_4(8_{20})=0$, because the knot is also smoothly slice, and its maximal self-linking number $\SL(8_{20})$ is equal to $-1$;
moreover, a computation of knot Floer homology shows that the unique top Alexander grading homology class has bigrading $(M,A)=(2,2)$.
\begin{figure}[ht]
 \centering
 \includegraphics[width=6cm]{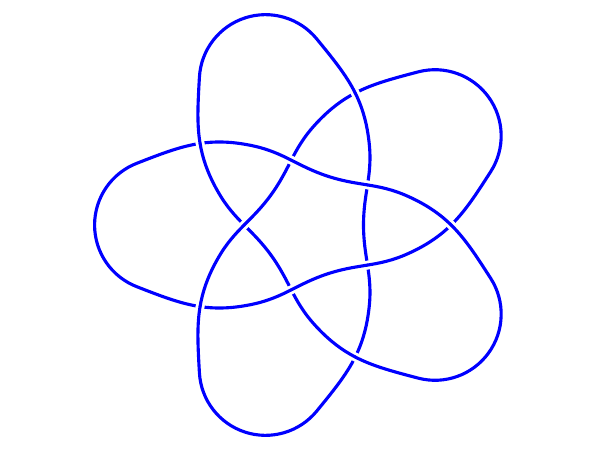}
 \hspace{2cm}
 \includegraphics[width=7cm]{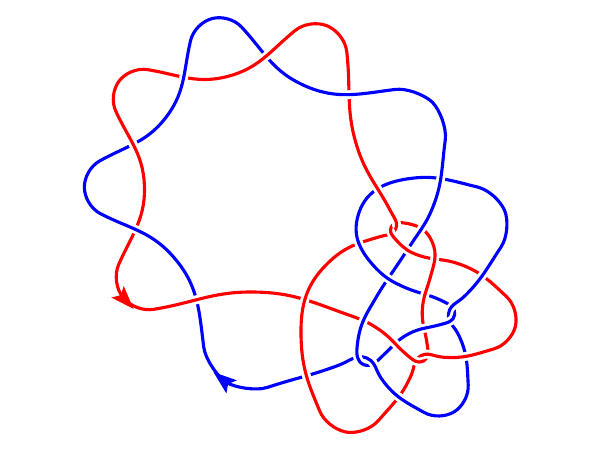}
 \caption{The knot $8_{20}$ (left) and the link $L=(6_1)_{2,-10}$ (right).}
 \label{8_20}
\end{figure}
We know then that there is a non-loose Legendrian knot $\mathcal L_1$ with $\Leghat(\mathcal L_1)\neq 0$ (whose transverse push-off is also non-loose) in $(S^3,\xi)$ where $d_3(\xi)=2A-M=2$. 
On the other hand, according to Theorem \ref{thm:main} we find another non-loose knot $\mathcal L_2$ with (at least) $\EH(\mathcal L_2)\neq 0$ which lives in $(S^3,\eta)$ with $d_3(\eta)=-\SL(8_{20})=1$. This implies that the two realizations are really distinct, as they appear in different contact structures.

\subsection{Links and sublinks}
Corollary \ref{cor:sublink} extends the applicability of Theorem \ref{thm:main} to links which themselves do not satisfy conditions of the theorem, when any of their sublinks does. 
In contrast, we want to show that there are also links that satisfy the hypothesis of Theorem \ref{thm:main}, but none of their sublinks do. 
In order to accomplish this, we consider the 2-component link $L$ defined as the $(2,-10)$-cable of the Stevedore's knot $6_1$, see Figure \ref{8_20}, with oppositely oriented components. 

The components of $L$ are two knots $6_1$: it was proved by Etnyre, Ng and Vértesi in \cite{Etnyre_twist} that $\TB(6_1)=\SL(6_1)=-5$, where $\TB$ refers to the maximal Thurston-Bennequin number, which means \[\slk(\mathcal K)\leq\SL(6_1)<-1=2\tau(6_1)-1\] for every Legendrian knot $\mathcal K$ with knot type $6_1$, and thus Corollary \ref{cor:sublink} is not useful. In addition, the knot $6_1$ is not fibered and the same is true for $L$ itself, since the only fibered links bounding an annulus in $S^3$ are the Hopf links. We make use of the following result of Rudolph from \cite{Rudolph}.
\begin{prop}[Rudolph]
 The $(2,2k)$-cable of a non-trivial knot $K$ in $S^3$ is strongly quasi-positive, and then quasi-positive, if and only if $k\leq\TB(K)$.
\end{prop}
Since in our case $k=-5=\TB(K)$, the link $L$ is strongly quasi-positive; hence, it satisfies the (classical) Thurston-Bennequin equality. But since it is not fibered, we do not a priori know whether taut sutured hierarchy gives us a non-loose realization. However, there exists an $\mathcal L$ with $\Leg^+(\mathcal L)\neq 0$ (in fact, the sutured hierarchy results in one when it does not land in an overtwisted structure) and we succeed in finding a non-loose Legendrian representative of $L$ with non-zero $\EH$-invariant in $d_3=-2\tau(L)+2=0$, thanks to Theorem \ref{thm:main}. Additionally, by Remark \ref{rmk:sub} there are non-loose realizations of $L$ with non-zero $\EH$-invariant also in the contact structure where $\pi$-torsion is applied only to a single component, that is, in $d_3=-\SL(6_1)=5$.

\end{document}